\documentclass{amsart}
\usepackage{amssymb,amsmath,amsthm,enumerate,amsbsy}
\usepackage[mathscr]{eucal}
\usepackage{framed,color,graphicx}
\usepackage{mathrsfs}
\usepackage{cite}
\usepackage[all]{xy}
\usepackage{tikz}
\usetikzlibrary{positioning,decorations.pathreplacing,patterns,decorations.pathmorphing}
\tikzset{%
element/.style={draw, shape=circle, fill=white, inner sep=1.4pt}
}

\usepackage[pdftex,
            pdfauthor={Zidong Gao},
            pdftitle={A finitely based finite semiring generates a variety with continuum many subvarieties},
            pdfsubject={Mathematics},
            pdfkeywords={semiring, variety, finite basis problem},
            bookmarksnumbered=true,
            pdfstartview={FitH},
            colorlinks=true,
            citecolor=blue,
            linkcolor=red,
            urlcolor=cyan]{hyperref}

\usepackage[utf8]{inputenc}
\usepackage[T1]{fontenc}

\DeclareSymbolFont{bbold}{U}{bbold}{m}{n}
\DeclareSymbolFontAlphabet{\mathbbold}{bbold}

\theoremstyle{plain}
\newtheorem{theorem}{Theorem}[section]
\newtheorem{lemma}[theorem]{Lemma}
\newtheorem{corollary}[theorem]{Corollary}
\newtheorem{proposition}[theorem]{Proposition}

\newtheorem{problem}[theorem]{Problem}

\theoremstyle{definition}

\newtheorem{question}[theorem]{Question}
\newtheorem{remark}[theorem]{Remark}

\newcommand{\occ}{\operatorname{occ}}

\renewcommand{\ge}{\geqslant}
\renewcommand{\le}{\leqslant}

\newcommand{\Mod}{\mathsf{Mod}}

\newcommand{\bp}{\mathbf{p}}
\newcommand{\bq}{\mathbf{q}}

\newcommand{\bu}{\mathbf{u}}
\newcommand{\bv}{\mathbf{v}}
\newcommand{\bw}{\mathbf{w}}

\begin{document}

\title
[A finitely based finite semiring generates a variety with continuum many subvarieties]
{A finitely based finite semiring generates a variety with continuum many subvarieties}

\author{Zidong Gao}
\address{School of Mathematics, Northwest University, Xi'an, 710127, Shaanxi, P.R. China}
\email{zidonggao@yeah.net}
%
%
%
%
\subjclass[2010]{16Y60, 03C05, 08B15}
\keywords{semiring, variety, finite basis problem}
\thanks{Zidong Gao is supported by National Natural Science Foundation of China (12371024, 12571020).}

\begin{abstract}
This paper establishes the existence of a finitely based finite semiring whose variety contains a continuum of subvarieties; such a variety is said to be of type \(2^{\aleph_0}\). Using the homomorphism theory of Kneser graphs, we prove that the 3-element semiring \(S_{53}\) is the first known example with this property. Moreover, \(S_{53}\) belongs to the variety of the max-plus semiring \((\mathbb{N},\max,+)\), which therefore is also of type \(2^{\aleph_0}\). For the finitely based 4-element semiring \(B_0\), we demonstrate that its variety contains infinitely many subvarieties and suggest that \(B_0\) could be another potential example of type \(2^{\aleph_0}\).
\end{abstract}

\maketitle

\section{Introduction}\label{sec:intro}

A \emph{variety} is a class of algebras closed under taking subalgebras, homomorphic images, and arbitrary direct products.
For a class $\mathcal{K}$ of algebras,
we write $\mathsf{V}(\mathcal{K})$ for the variety generated by $\mathcal{K}$,
that is, the smallest variety containing $\mathcal{K}$.
In particular, $\mathsf{V}(A)$ denotes the variety generated by a single algebra $A$.

By Birkhoff's theorem, a class of algebras is a variety if and only if it is an \emph{equational class},
that is, it consists of all algebras satisfying a certain set of identities.
For any class $\mathcal{K}$ of algebras, denote by $\mathsf{Id}(\mathcal{K})$ the set of all identities over a given countably infinite set $X$ of variables that are satisfied by $\mathcal{K}$.
A variety is \emph{finitely based} if it can be defined by a finite set of identities;
otherwise, it is \emph{nonfinitely based}.
Correspondingly, an algebra $A$ is called finitely based (nonfinitely based) if the variety $\mathsf{V}(A)$ generated by $A$ is finitely based (nonfinitely based).
The finite basis problem for a class of algebras,
one of the most important problems in universal algebra,
concerns the classification of its members according to whether they are finitely based.

For a variety \(\mathcal{V}\), the collection of all its subvarieties forms a complete lattice with respect to class inclusion, denoted by \(\mathcal{L}(\mathcal{V})\) and called the \emph{subvariety lattice} of \(\mathcal{V}\). If \(\mathcal{W}\) is a subvariety of \(\mathcal{V}\), then the interval \([\mathcal{W},\mathcal{V}]\) --- consisting of all varieties between \(\mathcal{W}\) and \(\mathcal{V}\) --- is a complete sublattice of \(\mathcal{L}(\mathcal{V})\).
Understanding the structure of the subvariety lattice $\mathcal{L}(\mathcal{V})$ is a fundamental problem in the study of varieties.
A central aspect of this problem is determining the size of intervals within the lattice.
Such descriptions reveal not only the internal complexity of $\mathcal{V}$,
but also often reflect important algebraic and logical properties of the algebras it contains.

A \emph{semiring} is an algebraic structure $(S, +, \cdot)$ equipped with two binary operations $+$ and $\cdot$ such that:
\begin{itemize}
\item the additive reduct $(S, +)$ is a commutative semigroup;
\item the multiplicative reduct $(S, \cdot)$ is a semigroup;
\item the distributive laws hold:
\[
(x+y)z \approx xz + yz, \qquad x(y+z) \approx xy + xz.
\]
\end{itemize}
An  \emph{additively idempotent semiring} (ai-semiring for short)
is a semiring whose additive reduct is idempotent.
This class of algebras has received particular attention in the study of varieties and equational properties.  Such algebras are ubiquitous in mathematics and find applications in diverse
areas as seen in fields such as algebraic geometry~\cite{cc}, tropical geometry~\cite{ms}, information science~\cite{gl}, and theoretical computer science~\cite{go}.


Let $S$ be an ai-semiring. Then the relation $\leq$ on $S$ defined by
\[
a \leq b \Leftrightarrow a+b=b
\]
is a partial order such that $(S, \leq)$ forms an upper semilattice, where the supremum of any two elements $a$ and $b$ is $a+b$.
Consequently, the additive reduct $(S, +)$ is uniquely determined by this semilattice order.
It is therefore often convenient to visualize the addition via the Hasse diagram of $(S, \leq)$.
Moreover, the order $\leq$ is readily seen to be compatible with multiplication,
which explains why such an algebra is also termed a \emph{semilattice-ordered semigroup}.

Over the past two decades, the finite basis problem and other variety-theoretic properties
for ai-semirings have been intensively studied and well developed
(see~\cite{dol07, dgv, dol09, gjrz, gv2302, gv2301,gv2501, gpz, jrz, pas05, rlyc, rlzc, rjzl, rzv, rz16, rzw, sr, vol21, yrzs, zrc}).
In particular,
Pastijn et al.~\cite{gpz, pas05} established that there are precisely $78$ ai-semiring
varieties satisfying the identity $x^2\approx x$, all of which are finitely based.
Ren et al.~\cite{rz16, rzw} later proved that there are precisely $179$ ai-semiring
varieties satisfying the identity $x^3\approx x$, which are likewise all finitely based.
Recently, Volkov et al.~\cite{rzv} showed that for any integer $n\geq 4$,
there are $2^{\aleph_0}$ distinct varieties satisfying $x^n \approx x$,
most of which are nonfinitely based.

\begin{table}[ht]
\centering
\caption{The Cayley tables of $S_7$} \label{tbs7}
\begin{tabular}{c|ccc}
$+$      &$0$&$a$&$1$\\
\hline
$0$      &$0$&$0$&$0$\\
$a$      &$0$&$a$&$0$\\
$1$      &$0$&$0$&$1$\\
\end{tabular}\qquad\qquad
\begin{tabular}{c|ccc}
$\cdot$  &$0$&$a$&$1$\\
\hline
$0$      &$0$&$0$&$0$\\
$a$      &$0$&$0$&$a$\\
$1$      &$0$&$a$&$1$\\
\end{tabular}
\end{table}

Shao and Ren~\cite{sr} established that every ai-semiring in the variety generated by
all ai-semirings of order two is finitely based.
Zhao et al.~\cite{zrc} showed that there are, up to isomorphism, $61$
ai-semirings of order three, which are denoted by $S_i$, $1 \leq i \leq 61$.
All of these are finitely based, with the possible exception of the semiring $S_7$; see Table~\ref{tbs7} for its Cayley tables.
Jackson et al.~\cite{jrz} later established that $S_7$ itself is nonfinitely based.
In fact, they showed that $S_7$ can transmit this property to many other finite ai-semirings.
Such examples include finite flat semirings whose varieties contain $S_7$,
and the ai-semiring $B^1_2$ whose multiplicative reduct is the six-element Brandt monoid
\[
\begin{array}{cccccc}
\begin{pmatrix}
0 & 0 \\
0 & 0
\end{pmatrix} &
\begin{pmatrix}
1 & 0 \\
0 & 1
\end{pmatrix} &
\begin{pmatrix}
0 & 1 \\
0 & 0
\end{pmatrix} &
\begin{pmatrix}
0 & 0 \\
1 & 0
\end{pmatrix} &
\begin{pmatrix}
1 & 0 \\
0 & 0
\end{pmatrix} &
\begin{pmatrix}
0 & 0 \\
0 & 1
\end{pmatrix} \\
0 & 1 & e_{12} & e_{21} & e_{11} & e_{22}
\end{array}
\]
with multiplication given by ordinary matrix multiplication.
Its additive reduct is determined by the Hasse diagram in Figure~\ref{fig1}.
Perkins~\cite{perkins1969} showed that the multiplicative reduct of $B^1_2$ is nonfinitely based;
indeed, it was the first example of a nonfinitely based finite semigroup,
and is sometimes referred to as Perkins' semigroup.

One can readily verify that $B^1_2$ contains a copy of $S_7$, whence $\mathsf{V}(S_7)$ is a subvariety of $\mathsf{V}(B^1_2)$. Recently, Gao et al. further showed that the interval $[\mathsf{V}(S_7), \mathsf{V}(B_2^1)]$ contains continuum many subvarieties, each of which is nonfinitely based.
\begin{figure}[h]\label{b21}
\centering
\scalebox{0.8}{%
\begin{tikzpicture}[
    node distance=1.5cm and 1.5cm,
    solidnode/.style={circle, draw=black, fill=black, inner sep=2pt, minimum size=4pt},
    hollownode/.style={circle, draw=black, fill=white, inner sep=2pt, minimum size=4pt}
]

\node[solidnode, label=below:1] (one) at (0,0.5) {};

\node[hollownode] (e12) at (-2,2) {};
\node[solidnode] (e11) at (-0.8,2) {};
\node[solidnode] (e22) at (0.8,2) {};
\node[hollownode] (e21) at (2,2) {};

\node[below=2pt of e12] {$e_{12}$};
\node[below left=2pt and -8pt of e11] {$e_{11}$}; 
\node[below right=2pt and -8pt of e22] {$e_{22}$}; 
\node[below=2pt of e21] {$e_{21}$};

\node[solidnode, label=above:0] (zero) at (0,3.5) {};

\draw (one) -- (e11);
\draw (one) -- (e22);

\draw (e12) -- (zero);
\draw (e11) -- (zero);
\draw (e22) -- (zero);
\draw (e21) -- (zero);

\end{tikzpicture}%
}
\caption{The additive order of $B_2^1$}
\label{fig1}
\end{figure}
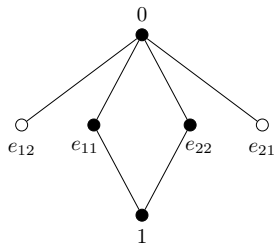

On the other hand, several intervals and subvariety lattices of ai-semiring varieties have been shown to have the cardinality of the continuum.
In particular, Dolinka~\cite{dol08} proved that the interval $[\mathsf{V}(\Sigma_7), \mathsf{V}(\mathcal{R}_2)]$ also has cardinality $2^{\aleph_0}$.
Here, $\Sigma_7$ denotes the $7$-element ai-semiring obtained from $B_2^1$ by adjoining a new element that acts as both the additive least element and the multiplicative zero, while $\mathcal{R}_2$ is the semiring of all binary relations on a $2$-element set.

To the best of our knowledge, $\mathcal{R}_2$ is the first example of a finite ai-semiring whose variety contains $2^{\aleph_0}$ distinct subvarieties. It is also the first known example of an \emph{inherently nonfinitely based} semiring \cite{dol09}, meaning that every locally finite variety containing $\mathcal{R}_2$ is nonfinitely based. For context, $\Sigma_7$ itself was the first example of a nonfinitely based finite ai-semiring \cite{dol07}. We also note that Volkov~\cite{vol21} independently resolved the finite basis problem for $B^1_2$ and $\Sigma_7$ using a different method.
More recently, Shaprynski\u{\i} \cite{shap23} showed that the subsemiring $\{0,e_{11},e_{12},e_{22}\}$ of $B_2^1$, denoted $B_0$, is finitely based. Nevertheless, the finite basis problem for the flat subsemiring $B_2$ of $B_2^1$ remains open.

Recently, Gao et al.~\cite{gjrz} proved that the semiring $S_7$ is of type $2^{\aleph_0}$.
More precisely, \cite[Corollary 4.13]{gjrz} shows that for each $k \geq 3$, the interval $[\mathsf{V}(S_c(a_1\cdots a_k)), \mathbf{N}_{k+1}]$ in the subvariety lattice $\mathcal{L}(\mathsf{V}(S_7))$ has the cardinality of the continuum, where $\mathbf{N}_{k+1}$ denotes the subvariety of $\mathsf{V}(S_7)$ determined by the $(k+1)$-nilpotent identity.
Consequently, the variety $\mathbf{NF}_{k+1}$ generated by all $(k+1)$-nilpotent flat semirings is also of type $2^{\aleph_0}$.

The cases $\mathbf{NF}_1$ and $\mathbf{NF}_2$ are straightforward: $\mathbf{NF}_1$ is the trivial semiring variety $\mathbf{T}$, and $\mathbf{NF}_2$ coincides with $\mathsf{V}(S(a))$, so their subvariety lattices have cardinalities $1$ and $2$, respectively.
For $\mathbf{NF}_3$, Gao and Ren~\cite{gr} characterized its subdirectly irreducible members, which are precisely the graph semirings. Using cycle graph semirings, they further proved that $\mathbf{NF}_3$ is likewise of type $2^{\aleph_0}$.

Combined with recent results from \cite{gjrz2,gry}, it is now known that the subvariety lattice of $\mathsf{V}(\mathcal{R}_2)$ contains a chain of consecutive intervals
\[
[\mathbf{T},\mathsf{V}(S_7)],\
[\mathsf{V}(S_7),\mathsf{V}(B_2^1)],\
[\mathsf{V}(B_2^1),\mathsf{V}(\Sigma_7)],\
\text{and}\
[\mathsf{V}(\Sigma_7),\mathsf{V}(\mathcal{R}_2)],
\]
each of which has cardinality $2^{\aleph_0}$.

Several other known examples of varieties of type $2^{\aleph_0}$ exist in various algebraic contexts~\cite{dol08, jac:uncoutably, jaclee, gla, gus, rjzl, tra}. Notably, Jackson~\cite{jac:uncoutably} provided the first example of a finite, finitely based semigroup that generates such a variety.
For semirings, however, no such example had been found until now. This naturally leads to the question:

\begin{question}
Does there exist a finitely based finite semiring whose variety is of type $2^{\aleph_0}$?
\end{question}

In the present paper, we provide the first example of such a semiring.
Since every variety generated by a single $2$-element semiring contains only finitely many subvarieties (see~\cite{ww}),
we must therefore begin our search among $3$-element semirings.
By investigating the subvariety lattice of the variety generated by each finitely based $3$-element ai-semiring,
we discovered that most of these varieties turn out to have only finitely many subvarieties.

A notable exception is the variety generated by the $3$-element ai-semiring $S_{53}$, whose Cayley tables are given in Table~\ref{tb2}.

\begin{table}[ht]
\centering
\caption{Cayley tables of $S_{53}$} \label{tb2}
\begin{tabular}{c|ccc}
$+$ & $0$ & $a$ & $1$\\
\hline
$0$ & $0$ & $0$ & $0$\\
$a$ & $0$ & $a$ & $a$\\
$1$ & $0$ & $a$ & $1$\\
\end{tabular}
\qquad\qquad
\begin{tabular}{c|ccc}
$\cdot$ & $0$ & $a$ & $1$\\
\hline
$0$ & $0$ & $0$ & $0$\\
$a$ & $0$ & $0$ & $a$\\
$1$ & $0$ & $a$ & $1$\\
\end{tabular}
\end{table}

From Table~\ref{tb2} we see that $S_{53}$ shares the same multiplicative reduct with $S_7$, but its additive order is a chain. In fact, the additive order coincides with the divisibility relation of the multiplicative semigroup.
We shall prove that the variety generated by $S_{53}$ is of type $2^{\aleph_0}$.

The structure of this paper is as follows. In Section~\ref{sec:prelim}, we introduce the necessary notions and notation, particularly ideal quotients, flat semirings of the form \(S(W)\), and isoterms.
In Section~\ref{sec:blockhypergraph}, we present basic concepts related to hypergraphs and the ai-semiring terms induced by hypergraphs, and provide a sufficient condition for a variety to be of type \(2^{\aleph_0}\).
In Section~\ref{sec:division}, we show that the semigroup \(S_c(W)\) forms an ai-semiring under the partial order induced by the divisibility relation, denoted \(S_c^*(W)\).
Finally, using \(S_c^*(a_1\cdots a_n)\), we prove that the varieties \(\mathsf{V}(S_{53})\) and the max-plus semiring \((\mathbb N,\max,+)\) are of type \(2^{\aleph_0}\).
In Section~\ref{sec:b0}, we prove that the variety generated by the semiring \(B_0\) contains infinitely many subvarieties \(\{\mathsf{V}(S(a_1\cdots a_n))\}_{n\geq 1}\), which also suggests it may be of type \(2^{\aleph_0}\).

\section{Preliminaries}\label{sec:prelim}
This section presents the required preliminary material and notational conventions.
For further details on concepts such as free algebras and subalgebras, we refer the reader to~\cite{bs,jrz}.
We begin by introducing the following notation.
Let $P$ denote the set of all prime numbers. For any positive integer $n$, let $[n]$ denote the set $\{1,\dots,n\}$.
For any nonnegative integer $k$, let $\mathbb{N}_{\ge k}$ denote the set of integers greater than or equal to $k$. In particular, $\mathbb{N} = \mathbb{N}_{\ge 0}$.
Next, we will focus on introducing the concepts of flat semirings and ideal quotients.

A \emph{flat semiring} is an ai-semiring~\(S\) such that its multiplicative reduct has a zero~\(0\) and its addition satisfies \(a+b = 0\) for all distinct \(a, b \in S\).
As observed in~\cite[Lemma 2.2]{jrz} (see also~\cite[Lemma 4.1.1]{elk}), a semigroup with zero \(0\) becomes a flat semiring precisely when it is \(0\)-cancellative, that is, \(ab=ac\neq 0\) implies \(b=c\) and \(ab=cb\neq 0\) implies \(a=c\) for all \(a, b, c\in S\).

For any cancellative semigroup $(S,\cdot)$, let $S^0$ denote the semigroup obtained by adjoining a multiplicative zero $0$. Then $S^0$ is $0$-cancellative and, endowed with the corresponding flat addition, it becomes a flat semiring; we call it the \emph{flat extension} of $S$ and denote it by $\flat(S)$. We shall see that ideal quotients of the flat extension $\flat(X^+)$ of $X^+$ provide a rich source of flat semirings.

Suppose that \(S\) is an ai-semiring containing a proper nonempty subset \(J\) that is both a multiplicative absorbing ideal and an order-theoretic filter with respect to \(+\). Then we can form the \emph{ideal quotient} \(S/J\) by identifying all elements of \(J\). For example, the semiring \(S_7\) is obtained as the ideal quotient of \(\flat(\{a\}^*)\) with \(J = \{a^k \mid k > 1\}\).

Now let \(W\) be a set of words in \(X^+\). Denote by \(W^{\leq}\) the set of all (nonempty) subwords of words in \(W\). Define a semigroup on the set \(W^{\leq} \cup \{0\}\) by letting \(0\) be a multiplicative zero and, for \(\bu, \bv \in W^{\leq}\),
\[
\bu \cdot \bv =
\begin{cases}
\bu\bv, & \text{if } \bu\bv \in W^{\leq}, \\
0, & \text{otherwise}.
\end{cases}
\]
This semigroup is denoted by \(S(W)\). When the empty word \(1\) is allowed, the multiplicative reduct is a monoid, which we denote by \(M(W)\).
Working in the free commutative semigroup \(X_c^+\) (or monoid \(X_c^*\)) yields the commutative analogues \(S_c(W)\) and \(M_c(W)\).
When \(W\) is a singleton set \(\{\mathbf{w}\}\), we write \(\mathbf{w}\) in place of \(\{\mathbf{w}\}\).

All of the above constructions yield \(0\)-cancellative semigroups. By equipping them with the corresponding flat addition we obtain flat semirings, which we denote by the same symbols. They can all be realized as ideal quotients of flat extension semigroups. For example, the flat semiring \(S(W)\) is isomorphic to the ideal quotient of \(\flat(X^+)\) with respect to the set \(J := \{\bw \in X^+ \mid \bw \notin W^{\leq}\}\).

In Section~\ref{sec:division} we shall introduce another semilattice order on \(S_c(W)\) and \(M_c(W)\), turning them into ai-semirings. The resulting structures will be denoted by \(S_c^*(W)\) and \(M_c^*(W)\), respectively.

Let $X^+$ denote the free semigroup over a countably infinite set $X$ of variables.
By distributivity, all ai-semiring terms over $X$ can be expressed as a finite sum of words from $X^+$.
An \emph{ai-semiring identity} over $X$ is an
expression of the form
\[
\mathbf{u}\approx \mathbf{v},
\]
where $\mathbf{u}$ and $\mathbf{v}$ are ai-semiring terms over $X$.
Let $\mathbf{u}\approx \mathbf{v}$ be an ai-semiring identity over an alphabet $\{x_1, x_2, \ldots, x_n\}$.
We say that an ai-semiring \emph{$S$ satisfies $\mathbf{u}\approx \mathbf{v}$} or \emph{$\mathbf{u}\approx \mathbf{v}$ holds in $S$},
if $\mathbf{u}(a_1, a_2, \ldots, a_n)=\mathbf{v}(a_1, a_2, \ldots, a_n)$ for all $a_1, a_2, \ldots, a_n\in S$,
where $\mathbf{u}(a_1, a_2, \ldots, a_n)$ denotes the result of evaluating $\mathbf{u}$ in $S$
under the assignment $x_i\rightarrow a_i$, and similarly for $\mathbf{v}(a_1, a_2, \ldots, a_n)$.

Let $\mathbf{p}$ be a word in $X^+$ and $x\in X$ be a variable. Then $\occ(x, \mathbf{p})$ denotes the number of occurrences of $x$ in $\mathbf{p}$.
Also let $c(\mathbf{p})$ denote the set of all variables that occur in $\mathbf{p}$.
For an ai-semiring term $\mathbf{u}=\mathbf{u}_1+\cdots+\mathbf{u}_n$ with each $\mathbf{u}_i \in X^+$,
we shall use $c(\mathbf{u})$ to denote the set of all variables that occur in $\mathbf{u}$, that is,
\[
c(\mathbf{u})=\bigcup_{1\leq i \leq n}c(\mathbf{u}_i).
\]
Recall from \cite[Lemma 1.1]{sr} that an identity $\mathbf{u} \approx \mathbf{v}$ holds in $M_2$ (the flat semiring $M(1)$) exactly when $c(\mathbf{u}) = c(\mathbf{v})$; such identities are called \emph{regular}.

For any ai-semiring terms $\mathbf{u}$ and $\mathbf{v}$,
let $\mathbf{u} \preceq \mathbf{v}$ denote the identity $\mathbf{u} + \mathbf{v} \approx \mathbf{v}$.
Let $S$ be an ai-semiring.
We write $\mathbf{u} \preceq_{S} \mathbf{v}$ if $S$ satisfies the identity $\mathbf{u} \preceq \mathbf{v}$.
A word $\mathbf{w}$ is \emph{minimal} for $S$ if
$\mathbf{u} \preceq_{S} \mathbf{w}$ implies $\mathbf{u} = \mathbf{w}$ for every word $\mathbf{u}$.
Equivalently,
$\mathbf{w}$ is minimal for $S$ if and only if it is an \emph{isoterm} for $S$;
that is, $S$ satisfies no nontrivial identity of the form $\mathbf{w} \approx \mathbf{u}$ for any ai-semiring term $\mathbf{u}$.
If $\mathbf{w}$ is an isoterm for $S$,
then every word obtained from a subword of $\mathbf{w}$ by renaming its letters
is also an isoterm for $S$. The following lemma reveals a good connection between isoterms and flat semirings.

\begin{lemma}[{\cite[Lemma 5.6]{rjzl}}]\label{isosw}
Let $\mathbf{w}$ be a word and let $S$ be an ai-semiring.
Then $\mathbf{w}$ is an isoterm for $S$
if and only if the flat semiring $S(\mathbf{w})$ belongs to $\mathsf{V}(S)$.
\end{lemma}

\section{Hypergraph and hypergraph terms}\label{sec:blockhypergraph}
This section introduces hypergraphs and their induced ai-semiring terms, focusing on Kneser hypergraphs.
Using their homomorphism theory, we establish a sufficient condition for an ai-semiring variety to be of type $2^{\aleph_0}$.

A \emph{hypergraph} $\mathbb{H}$ is a pair $(V, E)$,
where $E$ is a family of nonempty subsets of a set $V$.
Each element of $V$ is a \emph{vertex} of $\mathbb{H}$,
and each element of $E$ is a \emph{hyperedge} of $\mathbb{H}$.
A subset of $V$ is a \emph{subhyperedge} of $\mathbb{H}$ if it is contained in some hyperedge.
For an integer $k \geq 1$, the hypergraph $\mathbb{H}$ is \emph{$k$-uniform} if $|e| = k$ for every $e \in E$.

A \emph{homomorphism} from a hypergraph $\mathbb{H}_1=(V(\mathbb{H}_1), E(\mathbb{H}_1))$ to a hypergraph
$\mathbb{H}_2=(V(\mathbb{H}_2), E(\mathbb{H}_2))$ is a mapping
$\varphi \colon V(\mathbb{H}_1) \to V(\mathbb{H}_2)$ such that for every hyperedge $e \in E(\mathbb{H}_1)$,
its image $\varphi(e)$ is a hyperedge in $E(\mathbb{H}_2)$.
We denote by ${\rm Hom}(\mathbb{H}_1, \mathbb{H}_2)$ the set of all homomorphisms from $\mathbb{H}_1$ to $\mathbb{H}_2$.

We now recall some ai-semiring terms associated with $k$-uniform hypergraphs,
which were introduced in \cite{aj, gjrz2, jrz, gjrz} and have found a series of successful applications.
Let $\mathbb H=(V, E)$ be a $k$-uniform hypergraph,
and let $\{x_v \mid v \in V\}$ be a set of variables in one-to-one correspondence with $V$.
We denote by ${\bf t}_{{\mathbb{H}}}$ the ai-semiring term
\[
\sum_{\{v_1,v_2,\ldots,v_k\} \in E} x_{v_1}x_{v_2}\cdots x_{v_k},
\]
and by $\mathbf{q}_{{\mathbb H}}$ the word
\[
\prod_{v \in V} x_v,
\]
where the product may be taken in any fixed order.
Note that each hyperedge of $\mathbb{H}$ gives rise to $k!$ distinct words in ${\bf t}_{{\mathbb{H}}}$.

Let $m$ and $k \geq 3$ be positive integers.
The \emph{$k$-uniform Kneser hypergraph} $\mathbb{H}_{k,m}=(V_{k,m}, E_{k,m})$
is the $k$-uniform hypergraph whose vertices are the $m$-subsets of the set $[km]$,
and whose hyperedges are the $k$ pairwise disjoint such subsets.
Consequently, each hyperedge corresponds to a partition of $[km]$ into $k$ parts of size $m$.

Let $\varphi$ be a mapping from $V_{k,m}$ to $V_{k,n}$. According to the definitions of hypergraph homomorphism and the Kneser hypergraph, $\varphi$ is a homomorphism from $\mathbb{H}_{k,m}$ to $\mathbb{H}_{k,n}$
if and only if $\varphi$ maps every set of $k$ pairwise disjoint $m$-subsets of $[km]$ to $k$ pairwise disjoint subsets in $[kn]$.
This condition is equivalent to requiring that $\varphi$ preserves disjointness, that is, it maps any two disjoint $m$-subsets of $[km]$ to disjoint $m$-subsets of $[kn]$.

Let $\mathbb{G}_{k, m}$ denote the \emph{Kneser graph}, whose vertices are the $m$-subsets of the set $[km]$,
with two $m$-subsets adjacent if they are disjoint.
The above argument shows that ${\rm Hom}(\mathbb{H}_{k,m},\mathbb{H}_{k,n}) = {\rm Hom}(\mathbb{G}_{k,m},\mathbb{G}_{k,n})$.
This result can also be found in \cite[Theorem 1]{bbdvv}.
On the other hand, \cite[Lemma 7.9.3]{cg} states that
${\rm Hom}(\mathbb{G}_{k,m},\mathbb{G}_{k,n})$ is nonempty if and only if $m$ divides $n$.
We therefore obtain:

\begin{lemma}\label{homgh}
Let $m, n$ and $k$ be integers with $k \geq 3$. Then
${\rm Hom}(\mathbb{H}_{k,m},\mathbb{H}_{k,n})$ is nonempty if and only if $m$ divides $n$.
\end{lemma}

\begin{lemma}\label{0-max}
Let $S$ be an ai-semiring whose multiplicative reduct is $S_c(\mathbf{w})$ with $\ell(\mathbf{w})\geq 2$.
Then
\begin{enumerate}
    \item[$(\rm i)$] $0$ is the greatest element of $S$.
    \item[$(\rm ii)$] For subwords $\mathbf{u},\mathbf{v}$ of $\mathbf{w}$, if $\mathbf{u}\leq \mathbf{v}$ then $\mathbf{u}$ is a subword of $\mathbf{v}$.
    \item[$(\rm iii)$] For any subset $X\subseteq S$, each element of the subsemiring $\langle X\rangle$ generated by $X$ is either $0$ or a superword of some word in $X$.
\end{enumerate}
\end{lemma}

\begin{proof}
We first prove $(\rm i)$. Assume, for the sake of contradiction, that $0$ is not the greatest element of $S$.
Then there exists a subword $\mathbf{w}_1$ of $\mathbf{w}$ such that $0\leq \mathbf{w}_1$.
Write $\mathbf{w}=\mathbf{w}_1\mathbf{w}_2$ for some $\bw_2\leq \bw$.
Since multiplication is order-preserving, we obtain
\[
0=0\cdot \mathbf{w}_2\leq \mathbf{w}_1\cdot \mathbf{w}_2=\mathbf{w}.
\]
On the other hand, using the distributive law we compute
\[
0 = \mathbf{w}_1\mathbf{w}
    = \mathbf{w}_1(\mathbf{w}+\mathbf{w}_2)
    = \mathbf{w}_1\mathbf{w} + \mathbf{w}_1\mathbf{w}_2
    = 0 + \mathbf{w}
    = \mathbf{w},
\]
which contradicts the assumption $\mathbf{w} \neq 0$. Hence $0$ must indeed be the greatest element of $S$, establishing $(\rm i)$.

To prove $(\rm ii)$, let $\mathbf{u}$ and $\mathbf{v}$ be two subwords of $\mathbf{w}$ with $\mathbf{u}\leq \mathbf{v}$.
We claim that $\mathbf{u}$ is a subword of $\mathbf{v}$.
Suppose the contrary; then there exists a letter $a$ such that $\occ(a,\mathbf{u}) > \occ(a,\mathbf{v})$.
Choose a positive integer $k$ such that in $S(\bw)$ we have $a^k\mathbf{u}=0$ but $a^k\mathbf{v}\neq 0$.
Because multiplication preserves order again, it follows that
\[
0 = a^k \cdot \mathbf{u} \leq a^k \cdot \mathbf{v} \neq 0,
\]
which contradicts $(\rm i)$ that $0$ is the greatest element of $S$.
Therefore $\mathbf{u}$ must be a subword of $\mathbf{v}$, proving $(\rm ii)$.

Finally, we prove $(\rm iii)$. Let $T$ denote the multiplicative subsemigroup generated by $X$.
Each element of $T$ is a product of finitely many elements from $X$, hence is either $0$ or a superword of some word in $X$.
Now the subsemiring $\langle X\rangle$ consists of finite sums (i.e., suprema) of elements from $T$.
By $(\rm i)$ and $(\rm ii)$, the partial order in $S$ coincides with the subword relation among comparable nonzero elements.
Therefore, any finite suprema of elements that are superwords of words from $X$ are themselves superwords of some word from $X$ (or equal $0$).
Consequently, every element of $\langle X\rangle$ is either $0$ or a superword of some word in $X$, which establishes $(\rm iii)$.
\end{proof}

\begin{remark}
The case $\ell(\mathbf{w}) = 1$ is indeed exceptional. According to \cite{sr}, the semigroup $S(a)$ admits two distinct additive orders that make it an ai-semiring, namely $N_2$ and $T_2$.
\end{remark}

According to the above lemma, divisibility is closely related to making the semigroup $S(\mathbf{w})$ into an ai-semiring by defining a semilattice order. We will discuss this further in Section~\ref{sec:division}.
Next, we present the main result of this paper. For convenience, let $P$ denote the set of all prime numbers.

\begin{theorem}\label{th1}
Let $\mathcal{V}$ be an ai-semiring variety.
Suppose that for each positive integer $n$,
there exists an ai-semiring $S_n$ in $\mathcal{V}$ whose multiplicative reduct is $S_c(a_1 \cdots a_n)$.
Then $\mathcal{V}$ is of type $2^{\aleph_0}$.
\end{theorem}

\begin{proof}
For each integer $k \ge 3$, we will show that  $\mathcal{NV}_{k+1}$  is of type $2^{\aleph_0}$, where $\mathcal{NV}_{k+1}$ is subvariety of $\mathcal{V}$ determined by the $(k+1)$-nilpotent identity.
For each prime number $p$, let $A_{k,p}$ denote the subsemiring of $S_{kp}$ generated by the set
\begin{equation}\label{generator}
        \{\,a_{i_1} \cdots a_{i_p} \mid \{i_1,\dots,i_p\} \text{ is a $p$-subset of } [kp] \,\}.
\end{equation}
By Lemma~\ref{0-max}, the zero element $0$ is the greatest element of $A_{k,p}$ and every nonzero element of $A_{k,p}$ is a subword of $a_1 \cdots a_{kp}$ whose length is at least $p$. Consequently,
\begin{equation}\label{akp-power}
        (A_{k,p})^k = \{0,\; a_1 \cdots a_{kp}\}, \qquad (A_{k,p})^{k+1} = \{0\}.
\end{equation}
In particular, $A_{k,p}$ is $(k+1)$-nilpotent, so $A_{k,p} \in \mathcal{NV}_{k+1}$.

Next, we shall use the family $\{A_{k,p}\}_{p\in P}$ to prove that $\mathcal{NV}_{k+1}$ contains $2^{\aleph_0}$ distinct subvarieties.
For each prime $q$, consider the Kneser hypergraph $\mathbb{H}_{k,q} = (V_{k,q},\, E_{k,q})$ and the associated identity $\sigma_{k,q}$ defined by
\begin{equation}\label{eq-sigma-kq}
    \mathbf{t}_{\mathbb{H}_{k,q}} \approx \mathbf{t}_{\mathbb{H}_{k,q}} + \mathbf{q}_{\mathbb{H}_{k,q}}.
\end{equation}
We will show that the identities $\sigma_{k,q}$ can distinguish the algebras $A_{k,p}$ from one another.

First we prove that $A_{k,q}$ does not satisfy the identity \eqref{eq-sigma-kq}.
Define an assignment $\varphi : \{x_v \mid v \in V_{k,q}\} \to A_{k,q}$ by
\[
        \varphi(x_v) = \prod_{i \in v} a_i .
\]
Notice that every $v \in V_{k,q}$ corresponds to a $q$-subset of $[kq]$; hence $\varphi(x_v)$ belongs to the generating set \eqref{generator} of $A_{k,q}$.
Consequently, for each hyperedge $\{v_1,\dots,v_k\} \in E_{k,q}$ we have $\varphi(x_{v_1}\cdots x_{v_k}) = a_1 \cdots a_{kq}$, and thus
\[
        \varphi\bigl(\mathbf{t}_{\mathbb{H}_{k,q}}\bigr)
        = \sum_{\{v_1,\dots,v_k\} \in E_{k,q}} \varphi(x_{v_1}\cdots x_{v_k})
        = a_1 \cdots a_{kq}.
\]
On the other hand, the word $\mathbf{q}_{\mathbb{H}_{k,q}}$ has length $\binom{kq}{q}$, which is greater than $k$; hence $\varphi(\mathbf{q}_{\mathbb{H}_{k,q}}) = 0$ by \eqref{akp-power}.
Since $0$ is the greatest element of $A_{k,q}$, we obtain
\[
        \varphi\bigl(\mathbf{t}_{\mathbb{H}_{k,q}}\bigr) = a_1 \cdots a_{kq}\neq 0=
        \varphi\bigl(\mathbf{t}_{\mathbb{H}_{k,q}}\bigr) + \varphi\bigl(\mathbf{q}_{\mathbb{H}_{k,q}}\bigr).
\]
Therefore $A_{k,q}$ does not satisfy $\sigma_{k,q}$.

Next we prove that if $p \neq q$, then $A_{k,p}$ does satisfy the identity \eqref{eq-sigma-kq}.
Consider an arbitrary assignment $\varphi : \{x_v \mid v \in V_{k,q}\} \to A_{k,p}$.
We shall show that $\varphi(\mathbf{t}_{\mathbb{H}_{k,q}}) = 0$, from which the identity follows trivially.

Suppose, for contradiction, that $\varphi(\mathbf{t}_{\mathbb{H}_{k,q}}) \neq 0$.
Since $0$ is the greatest element, this implies $\varphi(x_{v_1}\cdots x_{v_k}) \neq 0$ for every $\{v_1,\dots,v_k\} \in E_{k,q}$.
By \eqref{akp-power}, each nonzero product $\varphi(x_{v_1}\cdots x_{v_k})$ must equal $a_1 \cdots a_{kp}$.
Thus the set $\{\varphi(x_{v_1}),\dots,\varphi(x_{v_k})\}$ consists of $k$ disjoint subwords of $a_1\cdots a_{kp}$ of length $p$, whose product equals $a_1\cdots a_{kp}$.
Consequently, for each $u \in V_{k,q}$ there exists a unique $p$-subset $v_u \subseteq [kp]$ such that $\varphi(x_u) = \prod_{i \in v_u} a_i$.
Moreover, whenever $\{u_1,\dots,u_k\} \in E_{k,q}$, the $p$-subsets $v_{u_1},\dots,v_{u_k}$ are pairwise disjoint, i.e. $\{v_{u_1},\dots,v_{u_k}\} \in E_{k,p}$.
Hence the map $u \mapsto v_u$ defines a hypergraph homomorphism from $\mathbb{H}_{k,q}$ to $\mathbb{H}_{k,p}$.
This contradicts Lemma~\ref{homgh}.  Therefore $\varphi(\mathbf{t}_{\mathbb{H}_{k,q}}) = 0$, and $A_{k,p}$ satisfies $\sigma_{k,q}$.

To summarize, we have produced a family $\{A_{k,p}\}_{p \in P}$ of semirings in $\mathcal{NV}_{k+1}$ and a family $\{\sigma_{k,p}\}_{p \in P}$ of identities such that for any two primes $p,q \in P$,
\begin{equation}\label{akpsigmakp}
     A_{k,p} \models \sigma_{k,q}\Leftrightarrow p \neq q.
\end{equation}

Finally, by \cite[Proposition 1.3]{gry}, the mapping
\[
    \varphi\colon \mathcal{P}(P)\to \mathcal{L}(\mathcal{NV}_{k+1}),\quad Q\mapsto \mathsf{V}(\{A_{k,p}\mid p\in Q\})
\]
is a lattice embedding. This implies that the subvariety lattice $\mathcal{L}(\mathcal{NV}_{k+1})$ has at least $2^{\aleph_0}$ elements, i.e., $\mathcal{NV}_{k+1}$ is of type $2^{\aleph_0}$, as required.
\end{proof}

According to \cite[Proposition 2.6]{jrz}, every flat semiring \(S_c(a_1 \cdots a_n)\) belongs to the variety of \(S_7\).
Consequently, by Theorem~\ref{th1}, \(S_7\) contains \(2^{\aleph_0}\) distinct subvarieties.
This result itself also represents one of the main results established in \cite{gjrz}.

\section{The divisibility order on the semigroup $S_c(W)$}\label{sec:division}
In this section, we will show that the semigroup $S_c(W)$, under the divisibility relation, still forms an ai-semiring, denoted by $S_c^*(W)$, and further use $S_c^*(a_1\cdots a_n)$ and Theorem~\ref{th1} to demonstrate that $S_{53}$ is of type $2^{\aleph_0}$.

For each nonempty subset $W\subseteq X_c^+$, let $\tau$ be the divisibility relation on the semigroup $S_c(W)$; that is, $(x, y) \in \tau$ if and only if there exists $z \in M_c(W)$ such that $x = yz$.
It is easy to verify that $\tau$ is a partial order on $S_c(W)$, compatible with the multiplication, and under this partial order the multiplicative zero $0$ is the greatest element.
Furthermore, we will now show that the partial order induced by $\tau$ on $S_c(W)$ is a semilattice order,
and with respect to the addition defined by $x + y = \sup\{x, y\}$, it forms an ai-semiring.
To this end we consider a more general setting.

First consider the free commutative monoid with zero $(X_c^{+})^{0}$, equipped with the divisibility order.
It is straightforward to verify that this order is a semilattice order, compatible with multiplication,
and that $0$ is the greatest element with respect to the induced addition.
Indeed, for $\mathbf{u}, \mathbf{v} \in X_c^+$, write
\[
\mathbf{u}=x_1^{k_1}x_2^{k_2}\cdots x_m^{k_m}, \qquad
\mathbf{v}=x_1^{t_1}x_2^{t_2}\cdots x_m^{t_m},
\]
where $k_i, t_i \ge 0$ for $1\le i\le m$. Then it is easy to verify that
\[
\sup\{\mathbf{u}, \mathbf{v}\}=x_1^{\max\{k_1,t_1\}}x_2^{\max\{k_2,t_2\}}\cdots x_m^{\max\{k_m,t_m\}}.
\]
If we define addition by $\mathbf{u}+\mathbf{v}=\sup\{\mathbf{u}, \mathbf{v}\}$, then multiplication distributes over addition, and $(X_c^+, +, \cdot)$ becomes an ai-semiring.

Now for any nonempty subset $W$ of $X_c^+$, set $I_W = (X_c^+)^0\backslash W^{\leq}$, it is easy to see that $I_W$ is both a multiplicative absorbing ideal and an order-theoretic filter with respect to $+$.
The ideal quotient semiring $(X_c^+)^0/I_W$ will be denoted by $S^*_c(W)$.

By definition, it is easy to verify that the semirings $S_c^*(W)$ and $S_c(W)$ share the same multiplicative structure. The induced addition operation on $S^*_c(W)$ is given as follows:
$0$ is the absorbing element and for $\mathbf{u}, \mathbf{v} \in W^{\le}$,
\[
\mathbf{u}+\mathbf{v}=
\begin{cases}
\sup\{\mathbf{u}, \mathbf{v}\}, & \text{if } \sup\{\mathbf{u}, \mathbf{v}\}\in W^{\le},\\[4pt]
0, & \text{otherwise}.
\end{cases}
\]
The semilattice order induced by this addition coincides exactly with the divisibility relation $\tau$ on the semigroup $S_c(W)$. This proves that $S_c(W)$ forms an ai-semiring under the semilattice order with respect to $\tau$.
If we allow the empty word $1$ in this construction, then $1$ becomes both the additive and the multiplicative identity. In that case we use the notation $M_c^*(W)$.

\begin{remark}
For any commutative semigroup $S$, define a binary relation $\tau$ by $(a,b)\in \tau$ if and only if $b=ac$ for some $c\in S^1$.
It is routine to verify that $\tau$ is a compatible quasi-order on $S$. Moreover, $\tau$ is a partial order exactly when $S$ is $\mathcal{J}$-trivial (i.e., $S^1aS^1=S^1bS^1$ implies $a=b$ for all $a,b\in S$). This observation naturally raises two further questions:
\begin{enumerate}[$(\rm i)$]
    \item Under what conditions does $\tau$ induce a semilattice order on $S$?
    \item If such a semilattice order exists, what additional conditions ensure that the resulting structure becomes an ai-semiring, that is, when do the induced addition and multiplication satisfy the distributive law?
\end{enumerate}
\end{remark}
It is easy to see that $M_c^*(1)$ is isomorphic to $M_2$, $M_c^*(a)$ is isomorphic to $S_{53}$,
and $M_c^*(\{a\}^+)$ is isomorphic to the max-plus semiring $(\mathbb{N}\cup\{\infty\},\max,+)$,
where addition is defined as the maximum of numbers, multiplication as the ordinary multiplication of numbers, and $\infty$ is the absorbing element for both addition and multiplication.
Aceto et al.~\cite{aei} prove that the max-plus semiring $\mathbb N$ is nonfinitely based.
However, the members of \( \mathsf{V}(\mathbb N) \), as well as its minimal nonfinitely based subvarieties, are still not clearly characterized.
Using ideal quotient of $\mathbb N$ we obtain

\begin{lemma}\label{mc*ak}
The max-plus semiring $\mathbb N$ is isomorphic to a subdirect product of the family $\{M_c^*(a^k)\}_{k \geq 0}$.
Consequently, \( M_c^*(a^k) \) belongs to the variety \(\mathsf{V}(\mathbb N)\) for every \(k \geq 0\).
\end{lemma}

\begin{proof}
For each \(k \geq 0\), observe that \(\mathbb N_{\geq k}\) is both a multiplicative ideal and an order filter of \(\mathbb N\), and the quotient \(\mathbb N/\mathbb N_{\geq k}\) is isomorphic to \(M_c^*(a^k)\).
Since $\bigcap_{k \geq 0} \mathbb N_{\geq k} = \varnothing$, the natural map $\varphi \colon \mathbb N \to \prod_{k \geq 0} \mathbb N/\mathbb N_{\geq k}$ defined by $\varphi(n) = (n/\mathbb N_{\geq k})_{k \geq 0}$ is a subdirect embedding.
Hence, the statement follows.
\end{proof}

\begin{remark}\label{n=ninfty}
If we replace \(\mathbb N_{\geq k}\) by \(\mathbb N_{\geq k} \cup \{\infty\}\), it is easy to verify that the max-plus semiring \(\mathbb N \cup \{\infty\}\) is also isomorphic to a subdirect product of the family \(\{M_c^*(a^k)\}_{k \geq 0}\).
Consequently, the max-plus semirings \(\mathbb N\) and \(\mathbb N \cup \{\infty\}\) generate the same variety.
An alternative proof of this fact will be given at the end of this section.
\end{remark}

In what follows, we shall show that $S_{53}$ and  $\mathbb N$ are of type \( 2^{\aleph_0} \).
We first give an analogous version of \cite[Proposition 2.6]{jrz}, which states that the flat semiring \( S_{c}(a_1\cdots a_n) \) lies in the variety \( \mathsf{V}(S_7) \) for each \( n\geq 1 \).

\begin{lemma}\label{sc*a1ak}
The semiring \( S_c^*(a_1\cdots a_n) \) lies in the variety \( \mathsf{V}(S_{53}) \) for each \( n\geq 1 \).
\end{lemma}

\begin{proof}
For any positive integer \( n \), let \( S_n \) denote the direct product of \( n \) copies of \( S_{53} \), and let \( A_n \) be the subsemiring generated by the \( n \) elements
\[
\alpha_1 = (a,1,1,\ldots,1),\; \alpha_2 = (1,a,1,\ldots,1),\;\ldots,\; \alpha_n = (1,1,\ldots,1,a).
\]
Let \( J_n \) be the set of all elements of \( A_n \) that have a coordinate equal to \( 0 \). Then \( J_n \) is both a multiplicative ideal and an order-theoretic filter of \( A_n \).

A routine verification shows that the ideal quotient \( A_n/J_n \) consists of \( J_n \) together with singleton sets corresponding to each subword of \( a_1\cdots a_n \). Hence \( A_n/J_n \) is isomorphic to \( S_c^*(a_1\cdots a_n) \) under the natural homomorphism determined by \( a_i \mapsto \{\alpha_i\} \) for \( 1\le i\le n \).
Therefore, \( S_c^*(a_1\cdots a_n) \in \mathsf{V}(S_{53}) \), as required.
\end{proof}

\begin{theorem}\label{ths53}
For each $k\geq 3$, let $\mathbf{N}_{k+1}$ be the subvariety of $\mathsf{V}(S_{53})$ determined by the $(k+1)$-nilpotent identity. Then the interval $[\mathsf{V}(S_c^*(a_1\cdots a_k)),\mathbf{N}_{k+1}]$ has the cardinality of the continuum.
In particular, $S_{53}$ is a finitely based finite semiring of type $2^{\aleph_0}$.
\end{theorem}
\begin{proof}
First, by Lemma~\ref{sc*a1ak} we have $S_c^*(a_1\cdots a_n) \in \mathsf{V}(S_{53})$ for each positive integer $n$.
Then, by Theorem~\ref{th1}, the ai-semiring $S_{53}$ is of type $2^{\aleph_0}$.

Recall that in the proof of Theorem~\ref{th1} we used the family $\{A_{k,p}\}_{p\in P}$  to embed the power set $\mathcal{P}(P)$ as a lattice into $\mathbf{N}_{k+1}$.
To show that $[\mathsf{V}(S_c^*(a_1\cdots a_k)), \mathbf{N}_{k+1}]$ has the cardinality of the continuum for each $k\geq 3$, it suffices to prove that $S_c^*(a_1\cdots a_k)$ is a subalgebra of $A_{k,p}$ for every prime $p$.

Indeed, take the word $a_1 a_2 \cdots a_{kp}$ and choose $k$ pairwise disjoint subwords
\[
\mathbf{u}_1, \mathbf{u}_2, \dots, \mathbf{u}_k,
\]
each of length $p$. (For example, one may take $\mathbf{u}_i = a_{(i-1)p+1} a_{(i-1)p+2} \cdots a_{ip}$ for $i=1,\dots,k$.)
It is straightforward to verify that the subsemiring generated by $\{\mathbf{u}_1,\dots,\mathbf{u}_k\}$ inside $A_{k,p}$ is isomorphic to $S_c^*(a_1\cdots a_k)$; consequently, $S_c^*(a_1\cdots a_k)$ is a subalgebra of $A_{k,p}$.

Therefore, the lattice embedding constructed in Theorem~\ref{th1} restricts to an embedding of $\mathcal{P}(P)$ into the interval $[\mathsf{V}(S_c^*(a_1\cdots a_k)), \mathbf{N}_{k+1}]$, which forces this interval to have cardinality $2^{\aleph_0}$.
\end{proof}

Since \( S_c^*(a_1\cdots a_k) \) is a subsemiring of \( A_{k,p} \) for every \( p \in P \), it follows from~\eqref{akpsigmakp} that the identities \( \{\sigma_{k,p}\}_{p \in P} \) are all satisfied by \( S_c^*(a_1\cdots a_k) \). Moreover, this set of identities is \emph{irreducible within} \( \mathbf{N}_{k+1} \); that is, for each \( q \in P \), the set $\mathsf{Id}(\mathbf{N}_{k+1})$ together with $\{\sigma_{k,p} \mid p \in P \setminus \{q\}\}$ does not imply the identity $\sigma_{k,q}$.

This is because \( A_{k,q} \) belongs to \( \mathbf{N}_{k+1} \), satisfies every identity \( \sigma_{k,p} \) with \( p \neq q \), but fails to satisfy \( \sigma_{k,q} \).

This shows that different subsets of \( \{\sigma_{k,p}\}_{p \in P} \) define different subvarieties within $\mathbf{N}_{k+1}$.
Consequently, we obtain a lattice embedding

\begin{equation}\label{euqaembd}
\psi\colon \mathcal{P}(P) \rightarrow \bigl[\mathsf{V}(S_c^*(a_1\cdots a_k)), \mathbf{N}_{k+1}\bigr], \quad Q \mapsto \Mod(\sigma_{k,p} \mid p \in Q) \cap \mathbf{N}_{k+1}.
\end{equation}
Furthermore, since \( \mathbf{N}_{k+1} \) is itself finitely based, say \( \mathbf{N}_{k+1} = \Mod(\Sigma) \) for some finite set \( \Sigma \) of identities, we have
\[
\psi(Q) = \Mod\bigl( \Sigma \cup \{\sigma_{k,p} \mid p \in Q\} \bigr).
\]
By discarding redundant identities from \( \Sigma \), we can extract a subset \( \Sigma_Q \subset \Sigma \) such that \( \Sigma_Q \cup \{\sigma_{k,p} \mid p \in Q\} \) forms an irredundant basis for \( \psi(Q) \).

Let \( \mathcal{P}_{\mathrm{inf}}(P) \) denote the family of all infinite subsets of \( P \). Since \( \mathcal{P}_{\mathrm{inf}}(P) \) itself has the cardinality of the continuum, we have established the following result.
\begin{proposition}\label{syntax}
For every integer \( k \geq 3 \), the interval \( [\mathsf{V}(S_c^*(a_1\cdots a_k)), \mathbf{N}_{k+1}] \) contains continuum many varieties $\{\psi(Q)\}_{Q \in \mathcal{P}_{\mathrm{inf}}(P)}$, each admitting an infinite irredundant basis.
\end{proposition}
\begin{remark}
In fact, through the analysis above, we have obtained a general conclusion: whenever we have constructed a lattice embedding $\psi$ from the power set $\mathcal{P}(I)$ of some infinite set $I$ into a certain subvariety lattice via identities (as in~\eqref{euqaembd}), then the image under $\psi$ of every element of $\{\mathcal{P}_{{\rm inf}}(I)\}$ is nonfinitely based. Moreover, if the variety $\mathcal{V}$ itself is finitely based, then each of its images also admits an infinite irredundant basis of identities.
\end{remark}

By Lemma~\ref{mc*ak} and Theorem~\ref{ths53}, we immediately have
\begin{proposition}
The max-plus semiring $(\mathbb{N},\max,+)$ is of type $2^{\aleph_0}$.
\end{proposition}


At the end of this section, we conclude with the regularization of semirings.
Note that from the Remark~\ref{n=ninfty} the max-plus semirings $\mathbb N$ and $\mathbb N\cup \{\infty\}$ generates the same variety.
To investigate the relationship between them, we introduce the general situation.
For a semiring $S$, denote by $S^{\infty}$ the semiring obtained by adjoining a new element $\infty$, which is absorbing element for both addition and multiplication. The following observation relates the variety generated by $S^{\infty}$ to that generated by $S$ and $M_2$.

\begin{proposition}\label{reg-construction}
For any semiring $S$, $\mathsf{V}(S^{\infty}) = \mathsf{V}(S, M_2)$.
\end{proposition}

\begin{proof}
By construction, $S$ is a subsemiring of $S^{\infty}$. Moreover, let $\rho$ denote the equivalence relation corresponding to the partition $\{S,\{\infty\}\}$ of $S^{\infty}$. One easily verifies that $\rho$ is a congruence on $S^{\infty}$, and that the quotient $S^{\infty}/\rho$ is isomorphic to $M_2$. Hence $S$ and $M_2$ lie in $\mathsf{V}(S^{\infty})$, which implies $\mathsf{V}(S,M_2) \subseteq \mathsf{V}(S^{\infty})$.

For the reverse inclusion, assume that an identity $\mathbf{u} \approx \mathbf{v}$ holds in both $S$ and $M_2$. Then $c(\bu)=c(\bv)$. Let $\varphi : c(\mathbf{u}) \to S^{\infty}$ be an arbitrary assignment. If $\varphi(x) = \infty$ for some variable $x \in c(\mathbf{u})$, then the absorbing property of $\infty$ forces $\varphi(\mathbf{u}) = \varphi(\mathbf{v}) = \infty$. Otherwise, $\varphi$ takes all variables into $S$; viewing $\varphi$ as an assignment into $S$, the fact that $S$ satisfies $\mathbf{u} \approx \mathbf{v}$ gives $\varphi(\mathbf{u}) = \varphi(\mathbf{v})$. Thus $S^{\infty}$ satisfies every identity that holds in $S$ and $M_2$, whence $\mathsf{V}(S^{\infty}) \subseteq \mathsf{V}(S,M_2)$.

Consequently, $\mathsf{V}(S^{\infty}) = \mathsf{V}(S, M_2)$.
\end{proof}
Proposition~\ref{reg-construction} therefore implies that the identities satisfied by $S^{\infty}$ are precisely the regular identities satisfied by $S$.
In this sense, $S^{\infty}$ may be viewed as the \emph{regularization} of $S$. Also we have the following Corollary.
\begin{corollary}\label{sinfty=s}
For any smeiring $S$, $\mathsf{V}(S)=\mathsf{V}(S^{\infty})$ if and only if $M_2\in \mathsf{V}(S)$.
\end{corollary}
Combining Lemma~\ref{mc*ak} and Corollary~\ref{sinfty=s}, we obtain
\begin{proposition}
The max-plus semirings $\mathbb{N}\cup \{\infty\}$ and $\mathbb{N}$ generate the same variety of semirings.
\end{proposition}

%
In fact, according to \cite[Theorem 1]{ge}, if $M_2 \notin \mathsf{V}(S)$, then $\mathsf{V}(S^{\infty})$ covers $\mathsf{V}(S)$ in the semiring variety.
Hence, we have the following result.

\begin{proposition}
For any semiring $S$, there are at most two members in the interval $[\mathsf{V}(S),\mathsf{V}(S^{\infty})]$.
\end{proposition}

\section{The flat semiring $B_0$}\label{sec:b0}
In this section, we present another potential example of a finitely based finite semiring of type \(2^{\aleph_0}\),
the subsemiring \( B_0 = \{0, e_{11}, e_{12}, e_{22}\} \) of \(B_2^1\).
Shaprynski\u{\i} \cite{shap23} showed that \( B_0 \) is finitely based. By analyzing the multiplicative semigroup of \( B_0 \), he observed that for an assignment \(\varphi : X \to B_0\), the equality \(\varphi(\bp) = e_{12}\) holds if and only if \(\bp\) has the form \(\bp = \bp_1 x \bp_2\) for some \(\bp_1, \bp_2 \in X^*\) with \(\occ(x, \bp_1) = \occ(x, \bp_2) = 0\), and
\[
\varphi(y) =
\begin{cases}
e_{11}, & y \in c(\bp_1), \\
e_{12}, & y = x, \\
e_{22}, & y \in c(\bp_2).
\end{cases}
\]
This characterization directly yields the following lemma.

\begin{lemma}\label{lem:occ}
Suppose \( \bp, \bq \in X^+ \) satisfy \( \bp \preceq_{B_0} \bq \). Then \( \occ(x, \bq) = 1 \) implies \( \occ(x, \bp) = 1 \) for all \( x \in X \).
\end{lemma}
\begin{proof}
Suppose \( \occ(x, \bq) = 1 \) and write \( \bq = \bq_1 x \bq_2 \) with \( \bq_1, \bq_2 \in X^* \).
Define an assignment \(\varphi : X \to B_0\) by
\[
\varphi(y) =
\begin{cases}
e_{11}, & y \in c(\bq_1), \\
e_{12}, & y = x, \\
e_{22}, & y \in c(\bq_2),\\
0,      & \text{otherwise}.
\end{cases}
\]
Then \(\varphi(\bq) = e_{12}\). Since \(\varphi(\bp) \preceq \varphi(\bq) = e_{12}\) in the additive order of \(B_0\),
we must have \(\varphi(\bp) = e_{12}\) (as only \(e_{12}\) itself is below \(e_{12}\) in the order shown in Figure~\ref{b21}).

By the characterization mentioned above, \(\varphi(\bp) = e_{12}\) implies that \(\bp\) can be written as
\(\bp = \bp_1 x \bp_2\) for some \(\bp_1, \bp_2 \in X^*\) with \(\occ(x, \bp_1) = \occ(x, \bp_2) = 0\).
Consequently, \(\occ(x, \bp) = 1\), as required.
\end{proof}

\begin{proposition}\label{linb0}
Every linear word is an isoterm for \( B_0 \).
\end{proposition}

\begin{proof}
We only need to show that every linear word $x_1\cdots x_n$ is minimal for \( B_0 \). Suppose that \( \mathbf{p} \preceq_{B_0} x_1 \cdots x_n \) for some \( \mathbf{p} \in X^+ \). Since \( M_2 \) is isomorphic to the subsemiring $\{0,e_{11}\}$ of $B_0$, it follows that \(\mathbf{p} \preceq_{M_2} x_1 \cdots x_n \) and so \( c(\mathbf{p}) \subseteq \{x_1, \ldots, x_n\} \).

On the other hand, by Lemma~\ref{lem:occ}, \( \occ(x_i, \mathbf{p}) = 1 \) for all \( 1 \leq i \leq n \). Thus \( \mathbf{p} = x_{i_1} \cdots x_{i_n} \), where \(\{i_1, \ldots, i_n\} = \{1, \ldots, n\} \). If \( \mathbf{p} \neq x_1 \cdots x_n \) then there exists \( 1 \leq s < t \leq n \) such that \( i_s > i_t \). Define an assignment \(\varphi : X \to B_0\) such that
\[
\varphi(x) =
\begin{cases}
e_{11}, & x \in \{x_1, \ldots, x_{s-1}\}, \\
e_{12}, & x = x_s, \\
e_{22}, & x \in \{x_{s+1}, \ldots, x_n\}.
\end{cases}
\]
Then \(\varphi(x_1 \cdots x_n) = e_{12}\) but \(\varphi(\mathbf{p}) = 0\), a contradiction. Thus \( \mathbf{p} = x_1 \cdots x_n \), and so every linear word is an isoterm for \( B_0 \).
\end{proof}

By Lemma~\ref{isosw} and Proposition~\ref{linb0}, we immediately have
\begin{proposition}\label{prop:S(a1...an)inB0}
The flat semiring $S(a_1\cdots a_n)$ lies in the variety $\mathsf{V}(B_0)$ for each positive $n$.
\end{proposition}

\begin{corollary}\label{cor:infinitechain}
The subvariety lattice $\mathcal{L}(\mathsf{V}(B_0))$ is infinite. More precisely, it contains an infinite strictly ascending chain
\[
\mathsf{V}(S(a_1)) \subsetneq  \mathsf{V}(S(a_1a_2)) \subsetneq  \mathsf{V}(S(a_1a_2a_3))\subsetneq \cdots.
\]
\end{corollary}

\begin{proof}
For each positive integer $n$, it is easy to see that  $S(a_1\cdots a_n)$ is a subsemiring of $S(a_1\cdots a_{n+1})$.
Also, $S(a_1\cdots a_n)$ satisfies the $(k+1)$-nilpotent identity for \(k\ge n\), but $S(a_1\cdots a_{n+1})$ does not satisfy the $(n+1)$-nilpotent identity.
In conclusion, $\mathsf{V}(S(a_1\cdots a_n))$ is a proper subvariety of $\mathsf{V}(S(a_1\cdots a_{n+1}))$, so the result holds.
\end{proof}

\begin{remark}
If we could obtain a noncommutative version of Theorem~\ref{th1}, that is, replacing $S_c(a_1\cdots a_n)$ with $S(a_1\cdots a_n)$ while the conclusion still holds, then we would be able to conclude that $B_0$ is also a finitely based finite ai-semiring of type $2^{\aleph_0}$. This remains a topic for future research.
\end{remark}

\section{Conclusion}\label{sec:conclusion}

We have proved that the 3-element finitely based ai-semiring \(S_{53}\) generates a variety with continuum many subvarieties, providing the first known example of a finite, finitely based semiring of type \(2^{\aleph_0}\).
By means of ideal quotients, we also showed that \(S_{53}\) lies within the variety \(\mathsf{V}(\mathbb N)\); consequently, the max-plus semiring \(\mathbb N\) is likewise of type \(2^{\aleph_0}\).

Gao~et al.~\cite{gjrz} completely settled the finite basis problem for subvarieties of \(\mathsf{V}(S_7)\).
In particular, \(\mathsf{V}(S_c(abc))\) is the unique minimal nonfinitely based subvariety of \(\mathsf{V}(S_7)\) (see~\cite[Corollary 3.20]{gjrz}).
Although Proposition~\ref{syntax}, using syntactic methods, supplies continuum many subvarieties of \(\mathsf{V}(S_{53})\), each being nonfinitely based and possessing an infinite irredundant identity basis, we are still unable to specify concrete finite semirings in $\mathsf{V}(S_{53})$ that are nonfinitely based.
By analogy with the case of \(S_7\), the following question arises naturally:

\begin{problem}
Is \(\mathsf{V}(S_c^*(abc))\) the minimal nonfinitely based subvariety of \(\mathsf{V}(S_{53})\)?
\end{problem}

On the other hand, the finite basis problem for the semiring \(S_c(W)\), where \(W\) is a finite set of words, has been completely solved in~\cite{jrz,wzr}.
This naturally leads to the following related problem:

\begin{problem}
For which finite sets of words \(W\) is the semiring \(S_c^*(W)\) finitely based?
\end{problem}

Finally, we recall that for any semiring $S$, the variety $\mathsf{V}(S^{\infty})$ covers $\mathsf{V}(S)$ in the lattice of semiring varieties if and only if $M_2 \notin \mathsf{V}(S)$.
In particular, $S$ is finitely based whenever $S^{\infty}$ is finitely based.
However, the converse remains an open question:

\begin{problem}
Is the semiring $S^{\infty}$ finitely based whenever $S$ is finitely based?
\end{problem}

\subsection*{Acknowledgment}
The author expresses sincere gratitude to Professors Xianzhong Zhao and Miaomiao Ren for their valuable comments and suggestions on this work.

%
%
%
%
%

\end{document}